\documentclass[a4paper,11pt]{amsart}
\usepackage{amsmath,amsfonts,amssymb,amsthm,latexsym,amscd}
\usepackage[dvips]{graphicx}
\usepackage[T1]{fontenc}
\usepackage{parskip}
\def\A{\mathcal{A}}
\def\a{\alpha}
\def\b{\beta}
\def\g{\gamma}
\def\AG{\langle A,G\rangle}

\def\D{\mathcal{D}}
\def\Hom{\operatorname{Hom}}
\def\Im{\operatorname{Im}}

\def\P{\mathcal{P}}
\def\S{\Sigma}

\def\eps{\varepsilon}
\def\sig{\sigma}
\def\s{\operatorname{sign}}

\newtheorem{thm}{Theorem}[section]
\newtheorem{thm*}{Theorem}
\newtheorem*{rem*}{Remark}
\newtheorem{lem}[thm]{Lemma}

\newtheorem{cor}[thm]{Corollary}
\newtheorem{defn}[thm]{Definition}

\newtheorem{rem}[thm]{Remark}
\newtheorem{ex}[thm]{Example}

\newcommand\ris[5]{\raisebox{#1mm}{\hspace{#2mm}\includegraphics[width=#3mm]{#4.eps}\hspace{#5mm}}}

\begin{document}

\title{Invariants of closed braids via counting surfaces}

\author{Michael Brandenbursky}

\vspace{2cm}

\begin{abstract}
A Gauss diagram is a simple, combinatorial way to present a
link. It is known that any Vassiliev invariant may be obtained
from a Gauss diagram formula that involves counting subdiagrams of certain combinatorial types.
In this paper we present simple formulas for an infinite family of invariants in terms
of counting surfaces of a certain genus and number of boundary components in a Gauss diagram associated with a closed braid.
We then identify the resulting invariants with partial derivatives
of the HOMFLY-PT polynomial.
\end{abstract}

\maketitle

\section{Introduction.}

In this paper we consider link invariants arising from the  HOMFLY-PT polynomials. The HOMFLY-PT polynomial $P(L)$
is an invariant of an oriented link $L$ (see e.g. \cite{FYHLMO}, \cite{J}, \cite{LM}, \cite{PT}). It is a Laurent polynomial in two
variables $a$ and $z$, which satisfies the following skein relation:
\begin{equation*}
aP\left(\ris{-4}{-1}{10}{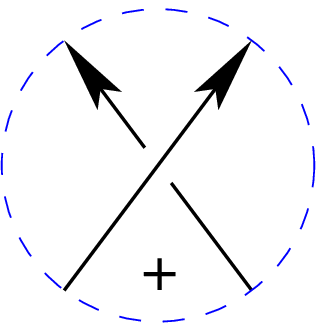}{-1.1}\right)-
a^{-1}P\left(\ris{-4}{-1}{10}{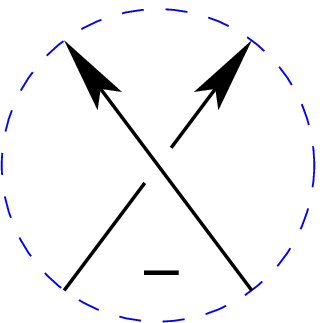}{-1.1}\right)=
zP\left(\ris{-4}{-1.1}{10}{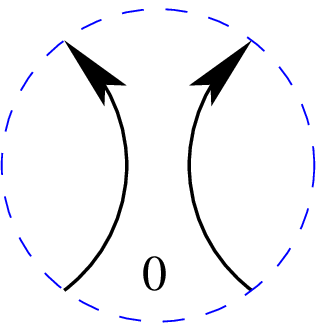}{-1.1}\right).
\end{equation*}
The HOMFLY-PT polynomial is normalized in the following way.
If $O_r$ is an $r$-component unlink, then
$P(O_r)=\left(\frac{a-a^{-1}}{z}\right)^{r-1}$. The Conway
polynomial $\nabla$ may be defined as $\nabla(L):=P(L)|_{a=1}$.
This polynomial is a renormalized version of the Alexander
polynomial (see e.g. \cite{Conway}, \cite{Likorish}). All coefficients
of $\nabla$ are finite type or Vassiliev invariants.

One of the mainstream and simplest techniques for producing
Vassiliev invariants are so-called Gauss diagram formulas (see
\cite{GPV}, \cite{PV}). These formulas generalize the calculation
of a linking number by counting subdiagrams of special
geometric-combinatorial types with signs and weights in a given
link diagram.

Until recently, explicit formulas of this type were known only for
few invariants of low degrees. The situation has changed with
works of Chmutov-Khoury-Rossi \cite{CKR} and Chmutov-Polyak
\cite{CP}. In \cite{CKR} Chmutov-Kho\-ury-Rossi presented an
infinite family of Gauss diagram formulas for all coefficients of
$\nabla(L)$, where $L$ is a knot or a two-component
link. Each formula for the coefficient $c_n$ of $z^n$ is related to a certain count of orientable
surfaces of a certain genus, and with one boundary component. The
genus depends only on $n$ and the number of the components of $L$.

In a recent paper \cite{B} the author showed that the $n$-th coefficient of the polynomial $zP^{(1)}_{a}(L)|_{a=1}$, where $P^{(k)}_{a}(L)|_{a=1}$ is the $k$-th partial derivative of the HOMFLY-PT polynomial $P$ w.r.t. the variable $a$ evaluated at $a=1$, can be obtained by a certain count of orientable surfaces of some genus with \textbf{one} and \textbf{two} boundary components. And again the
genus depends only on $n$ and the number of the components of $L$.

This leads to a natural question: how to produce link invariants
by counting orientable surfaces with an arbitrary number of
boundary components? In this paper we are going to show that the $n$-th coefficient of the polynomial $z^kP^{(k)}_{a}(L)|_{a=1}$ can be obtained by  a similar count of orientable surfaces of a certain genus with \textbf{one} up to $\textbf{k+1}$ boundary components, see Theorem \ref{thm:main}.

\textbf{Plan of the paper.} In Section 2 we review Gauss diagrams and Gauss diagram formulas. We define a notion of multi-based arrow diagrams and formulate our main result in terms of Gauss diagrams. In Section 3 we show that our invariant satisfies certain skein relation. In Section 4 we give a proof of our main Theorem \ref{thm:main}, and give an example. Section 5 is used for final remarks.

\textbf{Acknowledgments.} The author would like to thank Michael Polyak and Hao Wu for helpful conversations. We would like to thank the anonymous referee for careful reading of our paper and for his/her helpful comments and remarks. Part of this work has been done during the author's stay in Mathematisches Forschungsinstitut Oberwolfach. The author wishes to express his gratitude to the institute. He was supported by the Oberwolfach Leibniz fellowship.

%%%%%%%%%%%%%%%%%%%%

\section{Gauss diagrams and arrow diagrams}
\label{sec-Conway}
In this section we recall a notion of Gauss diagrams, arrow diagrams and Gauss diagram formulas. We then define a special type of arrow diagrams which will be used to define Gauss diagram formulas for coefficients of polynomials derived from the HOMFLY-PT polynomial.

\subsection{Gauss diagrams of links}
\label{subsec-coloring}

Gauss diagrams (see e.g.  \cite{GPV}, \cite{PV}) provide a simple combinatorial way to encode oriented links.

\begin{defn}\rm
Given a link diagram $D$, consider a collection of oriented circles parameterizing it. Unite two preimages of every crossing of $D$ in a pair and connect them by an arrow, pointing from the overpassing preimage to the underpassing one. To each arrow we assign a sign (local writhe) of the corresponding crossing. The result is called the \textit{Gauss diagram} $G$ corresponding to $D$.
\end{defn}

We consider Gauss diagrams up to an orientation-preserving diffeomorphisms of the circles. In figures we will always draw circles of the Gauss diagram with a counter-clockwise orientation. A classical link can be uniquely reconstructed from the corresponding Gauss diagram \cite{GPV}.

\begin{ex}\rm
Diagrams of the trefoil knot and the Hopf link, together with the corresponding Gauss diagrams, are shown in the following picture.
\begin{equation*}
\quad\ris{-8}{-3}{40}{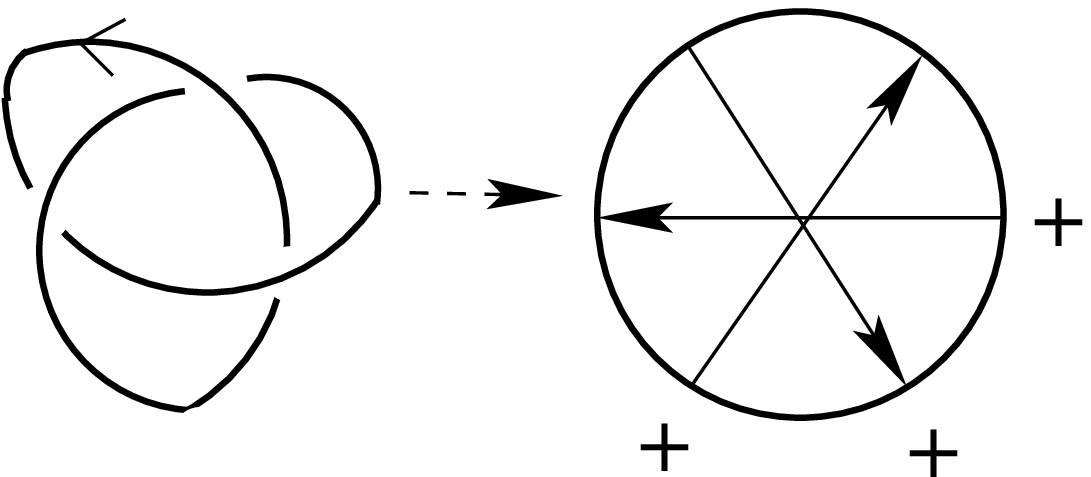}{-1.1}\quad\quad\quad\ris{-6.5}{-3}{74}{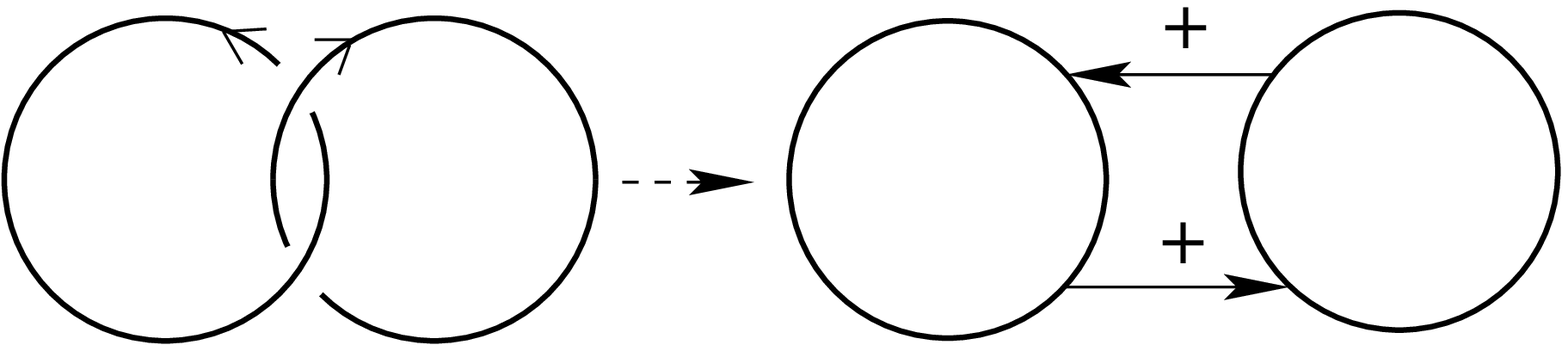}{-1.1}
\end{equation*}
\end{ex}

Two Gauss diagrams represent isotopic links  if and only if they are related by a finite number of Reidemeister moves for Gauss diagrams shown in Figure \ref{fig:Reidem}, where $\eps=\pm1$. See e.g. \cite{CDBook, Oestlund, P}.
\begin{figure}[htb]
\begin{equation*}
\Omega_{1}:\quad\ris{-7}{-3}{30}{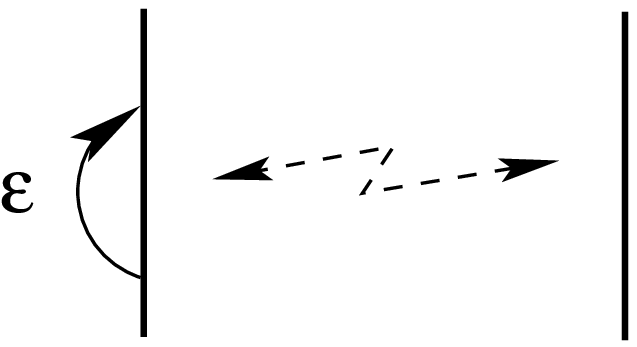}{-1.1}\quad\quad\quad\Omega_{2}:\quad\ris{-7}{-3}{45}{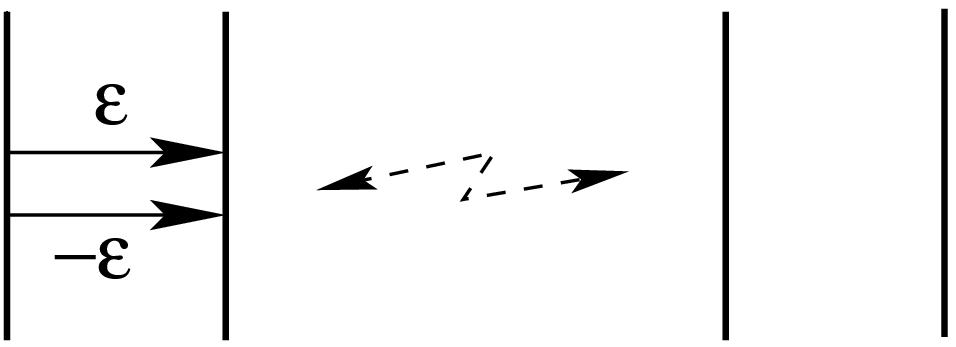}{-1.1}
\end{equation*}
\begin{equation*}
\Omega_{3}:\quad\ris{-8}{-3}{65}{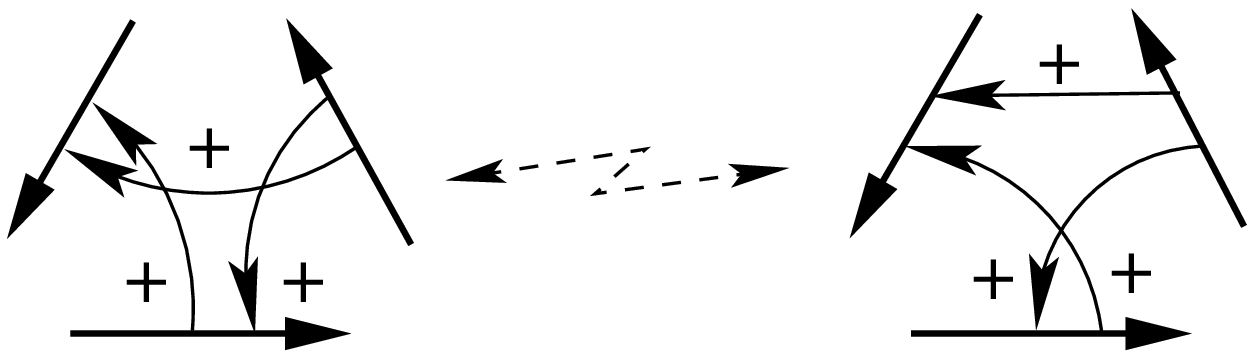}{-1.1}
\end{equation*}
\caption{\label{fig:Reidem} Reidemeister moves of Gauss diagrams.}
\end{figure}

\subsection{Gauss diagrams of closed braids}

Recall that the Artin braid group $B_m$ on $m$ strings has the following presentation:
\begin{equation*}
B_m=\langle\sig_1,\ldots,\sig_{m-1}|\hspace{2mm} \sig_i\sig_j=\sig_j\sig_i,\hspace{1mm}|i-j|\geq2;\hspace{2mm}\sig_i\sig_{i+1}\sig_i=\sig_{i+1}\sig_i\sig_{i+1}\rangle,
\end{equation*}
where each generator $\sig_i$ is shown in Figure \ref{fig:braid-gen-sig-i}a.
\begin{figure}[htb]
\centerline{\includegraphics[height=1.3in]{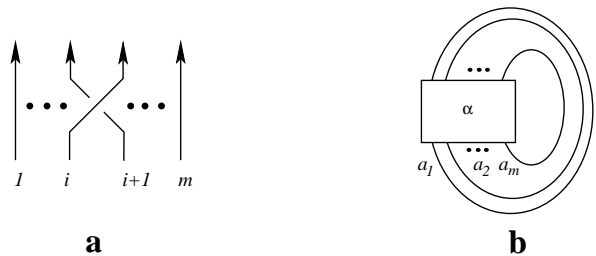}}
\caption{\label{fig:braid-gen-sig-i} Artin generator $\sig_i$ and a closure of a braid $\a$.}
\end{figure}
Let $w$ be a word on $m$ strings in generators $\sig_1^{\pm 1},\ldots,\sig_{m-1}^{\pm 1}$. We take the corresponding "geometric word", connect it opposite ends by nonintersecting curves as shown in Figure \ref{fig:braid-gen-sig-i}b and get an oriented link diagram $D$. Let us remove from $D$ a small neighborhood around each intersection point. We define an arc in $D$ to be a connected component of the resulting graph, and label the arcs of $D$ shown in Figure \ref{fig:braid-gen-sig-i}b by letters $a_1,...,a_m$. Two words $w$ and $w'$ on $m$ strings represent the same element in $B_m$ up to conjugation, if and only if the associated diagrams $D$ and $D'$ can be obtained one from another by a finite sequence of moves shown in Figure \ref{fig:braid-Rmoves}, see e.g. \cite{CT}. Note that the moves shown in Figure \ref{fig:braid-Rmoves} may also involve labeled arcs.

\begin{figure}[htb]
\centerline{\includegraphics[height=2.15in]{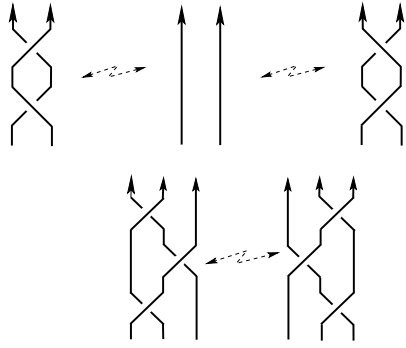}}
%\vspace{1in}
\caption{\label{fig:braid-Rmoves} Reidemeister moves of braid diagrams.}
\end{figure}

\begin{figure}[htb]
\begin{equation*}
\Omega_1:\hspace{18mm}\ris{-7}{-3}{85}{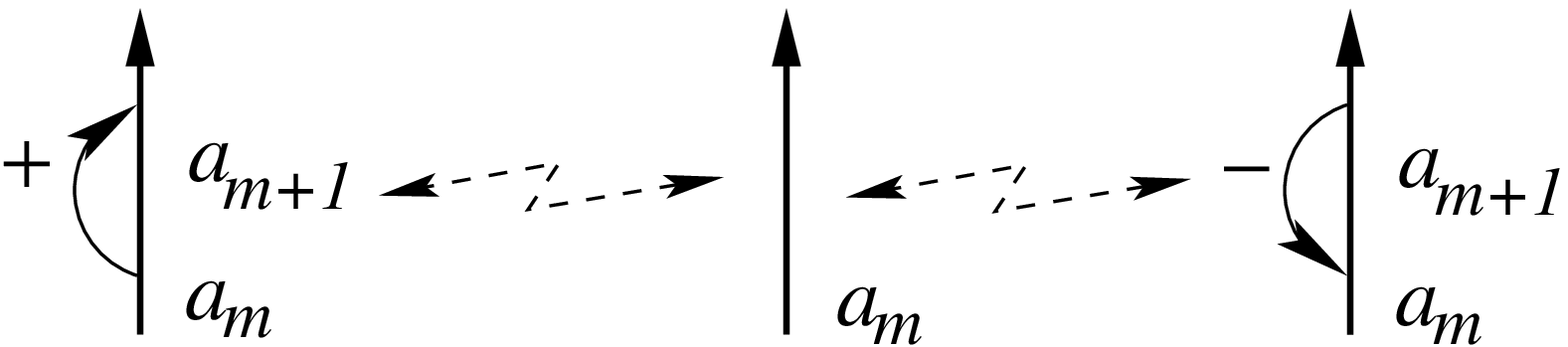}{-1.1}
\end{equation*}
\begin{equation*}
\Omega_2:\hspace{10mm}\ris{-35}{-3}{80}{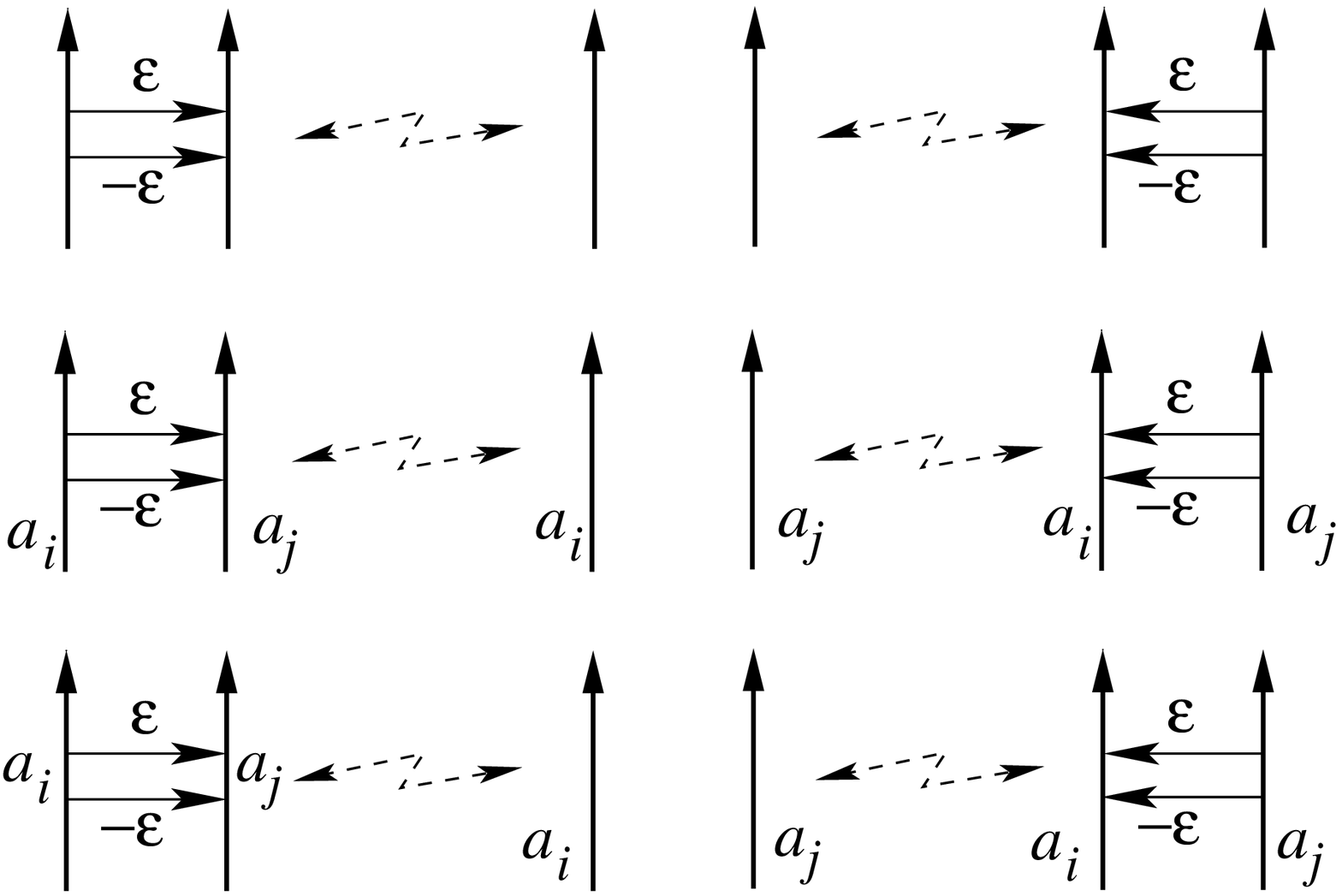}{-1.1}
\end{equation*}
\begin{equation*}
\Omega_3:\hspace{20mm}\ris{-8}{-3}{70}{reidemeister3}{-1.1}
\end{equation*}
\caption{\label{fig:braid-Gauss-Rmoves} Reidemeister moves of braid Gauss diagrams. The move $\Omega_2$ may involve arcs labeled by $a_i$ and $a_j$. In this case we require that $|i-j|=1$. The move $\Omega_3$ may involve labeled arcs as well, i.e. two corresponding arcs may be labeled by some $a_i$.}
\end{figure}

Let $D$ be any diagram associated with a braid $\a\in B_m$. A corresponding \emph{braid Gauss diagram} $G_m$ is a Gauss diagram $G$ together with the corresponding arcs labeled by letters $a_1,...,a_m$, where an arc in $G$ is a connected component of the complement of all arrows in $G$. Similarly to the case of Gauss diagrams, two braid Gauss diagrams represent isotopic links if and only if they are related by a finite number of moves shown in Figure~\ref{fig:braid-Gauss-Rmoves}.

\begin{defn}\label{defn:labeling}\rm
Let $k\geq 1$ be any integer and $G_m$ a braid Gauss diagram associated with a braid $\a\in B_m$. A \emph{colored braid Gauss diagram} $G_{k,m}$ is a diagram $G$ together with the following assignment of $k$ base points $*_1,...,*_k$ and $m-k$ natural numbers between $1$ and $k$ to the arcs labeled by letters $a_1,...,a_m$:
\begin{itemize}
\item
For each $1\leq i\leq k$ there exists \textbf{exactly one} arc $a_j$ such that the base point $*_i$ is placed on this arc, and the arc $a_1$ always contains basepoint $*_1$.
\item
Let $1\leq i_1<i_2\leq k$. If $*_{i_1}$ and $*_{i_2}$ are placed on arcs $a_{j_1}$ and $a_{j_2}$, then $j_1<j_2$, i.e. the assignment of base points is in ascending order.
\item
After we placed base points $*_1,...,*_k$ on arcs $a_{j_1},...,a_{j_k}$, let $a_j$ be a non-based arc. Denote by $j_l$ the maximal number from the set $\{j_1,...,j_k\}$ such that $j>j_l$. Now, to arc $a_j$ we assign \textbf{exactly one} number from the set $\{1,...,l\}$.
\end{itemize}
\end{defn}

Let us give an example of a colored braid Gauss diagram, in case when $k=3$, since the above definition seems to be complicated. In Figure \ref{fig:braid-Gauss-example} we show a braid diagram for $\sig_1\sig_2\sig_3\in B_4$, the corresponding braid Gauss diagram and an associated colored braid Gauss diagram $G_{3,4}$.

\begin{figure}[htb]
\centerline{\includegraphics[height=2.5in]{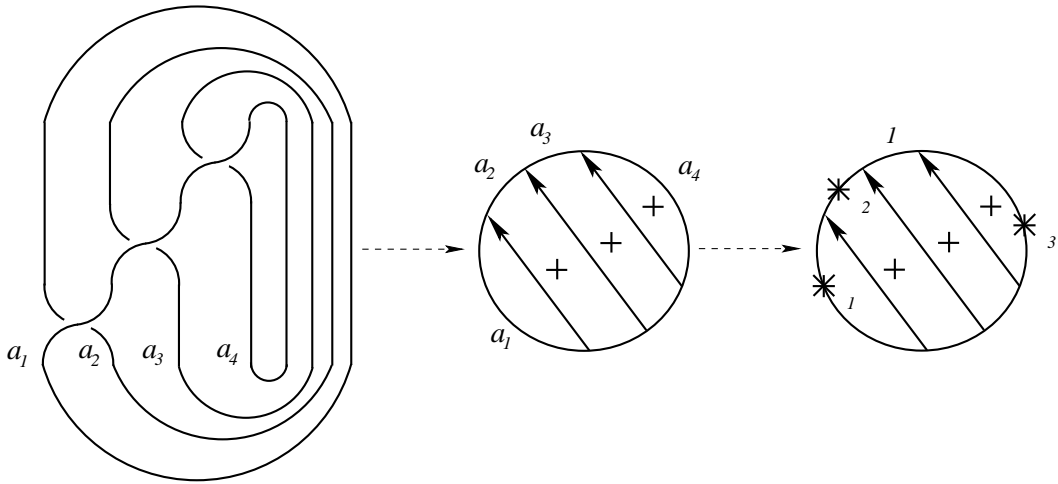}}
%\vspace{1in}
\caption{\label{fig:braid-Gauss-example} Braid diagram on the left, the corresponding braid Gauss diagram in the middle and an associated colored braid Gauss diagram $G_{3,4}$ on the right. A base point $*_1$ is placed on arc labeled by $a_1$, a base point $*_2$ is placed on arc labeled by $a_2$, a base point $*_3$ is placed on arc labeled by $a_4$, and so the assignment of the base points is in ascending order. Now we have to place a number from  $\{1,2,3\}$ on arc labeled by $a_3$. Since in this case $j_1=1$, $j_2=2$ and $j_3=4$, we can place any number from $\{1,2\}$ on arc labeled by $a_3$. In this picture we placed number 1.}
\end{figure}

We denote by $\mathfrak{G}_{k,m}$ a set of all colored braid Gauss diagrams associated with $k$ and $G_m$. Note that $\mathfrak{G}_{k,m}$ is empty whenever $k>m$.

\subsection{Arrow diagrams and the corresponding surfaces}

An \textit{arrow diagram} is a modification of a notion of a Gauss
diagram, i.e. it consists of a number of oriented circles with several arrows connecting pairs of distinct points on them, see Figure \ref{fig:non-realizable}. An arrow diagram is \emph{based} if a base point is placed on an arc of $A$. We consider these diagrams up to orientation-preserving diffeomorphisms of the circles.
\begin{figure}[htb]
\centerline{\includegraphics[height=0.65in]{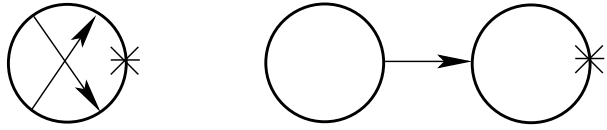}}
%\vspace{1in}
\caption{\label{fig:non-realizable}Based connected arrow diagrams.}
\end{figure}

Given an arrow diagram $A$, we define an oriented surface $\S(A)$ as
follows. Firstly, replace each circle of $A$ with an oriented disk
bounding this circle. Secondly, glue $1$-handles to boundaries of
these disks using each arrow as a core of an untwisted ribbon, such that the ribbons do not intersect in $\mathbb{R}^3$. See Figures
\ref{fig:surface} and \ref{fig:surface-const}.

\begin{figure}[htb]
\centerline{\includegraphics[height=0.65in]{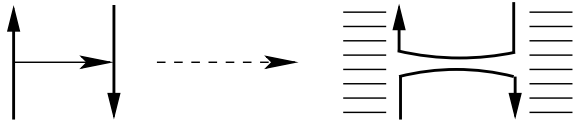}}
%\vspace{1in}
\caption{\label{fig:surface} Constructing a surface from an arrow diagram.}
\end{figure}

\begin{figure}[htb]
\centerline{\includegraphics[height=1.6in]{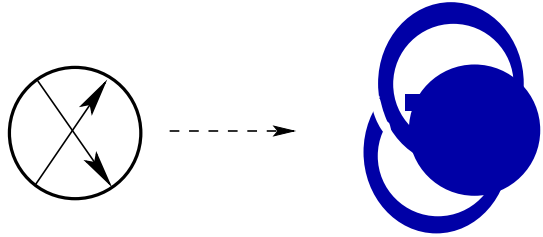}}
%\vspace{1in}
\caption{\label{fig:surface-const} Arrow diagram with 2 arrows and a corresponding surface, which is homeomorphic to the torus with one boundary component.}
\end{figure}

\begin{defn}\rm
By the \textit{genus} and the \textit{number of boundary components}
of an arrow diagram $A$ we mean the genus and the number of boundary
components of $\S(A)$. An arrow in $A$ is called \textit{separating}, if the boundary of the corresponding ribbon belongs to different boundary components of $\S(A)$.
\end{defn}

\begin{rem}\rm
Let $A$ be an arrow diagram with $n$ arrows and $r$ circles. Then the Euler characteristic $\chi$ of $\S(A)$ equals to $\chi(\S(A))=r-n$. If $A$ is connected, $n\geq r-1$. If $A$ has odd number of boundary components, $n\neq r (\rm{mod}2)$, otherwise $n=r (\rm{mod}2)$.
\end{rem}

\begin{ex}\rm
The arrow diagram with one circle in Figure \ref{fig:non-realizable} is of genus
one, while the other arrow diagram in the same figure is
of genus zero. Both of them have one boundary component.
\end{ex}

\begin{defn}\rm
An arrow diagram $A$ with $k$ boundary components is \emph{multibased} if $k$ base points $*_1,...,*_k$ are placed on $k$ different arcs of $A$, such that each arc belongs to a different boundary component of $A$, see Figure \ref{fig:multibased}. Let $1\leq j\leq k$. We say that a boundary component of $A$ is called $j$-th boundary component if there exists an arc, which belongs to this component, with a base point $*_j$ on it.
\end{defn}

\begin{figure}[htb]
\centerline{\includegraphics[height=0.65in]{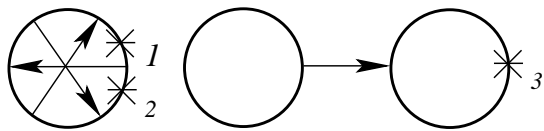}}
%\vspace{1in}
\caption{\label{fig:multibased} Multibased arrow diagram with 3 boundary components.}
\end{figure}

\subsection{Gauss diagram formulas}
\label{subsec-G-diag}

M. Polyak and O. Viro suggested \cite{PV} the following approach to
compute link invariants using Gauss diagrams.

\begin{defn}\label{defn:Arrow-Gauss-states}\rm
Let $A$ be a based arrow diagram with $r$ circles and let $G$ be a based
Gauss diagram of an $r$-component oriented link. A \textit{homomorphism}
$\phi:A\rightarrow G$ is an orientation preserving homeomorphism
between each circle of $A$ and each circle of $G$, which maps the base
point of $A$ to the base point of $G$ and induces an injective map
of arrows of $A$ to the arrows of $G$. The set of arrows in
$\Im(\phi)$ is called a \textit{state} of $G$ induced by $\phi$ and is
denoted by $S(\phi)$. The \textit{sign} of $\phi$ is defined as
$\s(\phi)=\prod_{\a\in S(\phi)}sign(\a)$. A set of all homomorphisms
$\phi:A\to G$ is denoted by $\Hom(A,G)$.
\end{defn}

Note that since the circles of $A$ are mapped to circles of $G$,
a state $S$ of $G$ determines both the arrow diagram $A$ and the
map $\phi:A\to G$ with $S=S(\phi)$.

\begin{defn}
\rm A \textit{pairing} between an  arrow diagram $A$ and $G$ is defined by
$$\AG=\sum_{\phi\in\Hom(A,G)}\s(\phi).$$
We set $\AG=0$, whenever $A$ and $G$ have different number of circles.
\end{defn}

For an arbitrary arrow diagram $A$ the pairing $\AG$ does not represent a link invariant, i.e. it depends on the choice of a Gauss diagram of a link. However, for some special linear combinations of arrow diagrams the result is independent of the choice of $G$. Using a slightly modified definition of arrow diagrams Goussarov, Polyak and Viro showed in \cite{GPV} that each real-valued Vassiliev invariant of knots may be obtained this way.
In other words, they showed that each real-valued Vassiliev invariant of knots may be computed as a certain count with weights of subdiagrams of a based Gauss diagram. For example, all coefficients of the Conway polynomial $\nabla$ may be obtained using suitable combinations of arrow diagrams. More precisely, in \cite{CKR} it was shown that the coefficient $c_n$ of $z_n$ in $\nabla$ can be obtained by a certain count of arrow diagrams with \textbf{one} boundary component and certain genus, where the genus depends only on $n$ and the number of circles in $G$.

In \cite{B} we showed that the $n$-th coefficient of the polynomial $zP^{(1)}_{a}|_{a=1}$ can be obtained by a certain count of arrow diagrams with \textbf{one} and \textbf{two} boundary components and certain genus. And again the genus depends only on $n$ and the number of circles in $G$. In what follows we are going to show that, in case when $G$ is a braid Gauss diagram of a link, the $n$-th coefficient of the polynomial $z^kP^{(k)}_{a}|_{a=1}$ can be obtained by  a similar count of arrow diagrams with \textbf{one} up to $\textbf{k+1}$ boundary components and a certain genus. Hence we need to adopt Definition \ref{defn:Arrow-Gauss-states} to the case of multi-based arrow diagrams and colored braid Gauss diagrams.

\begin{defn}\label{defn:MultiArrow-BraidGauss}\rm
Let $A$ be a multi-based arrow diagram with $r$ circles and $k$ boundary components and let $G_{k,m}$ be a colored braid
Gauss diagram of an $r$-component closed braid on $m$ strings. A \textit{homomorphism}
$\phi:A\rightarrow G_{k,m}$ is an orientation preserving homeomorphism
between each circle of $A$ and each circle of $G_{k,m}$, which maps each base
point $*_i$ of $A$ to each base point $*_i$ of $G_{k,m}$ and induces an injective map
of arrows of $A$ to the arrows of $G_{k,m}$. In addition we require that if a non-based arc $a$ of $G_{k,m}$ is labeled by some $j$, then $a$ is an image of some arc which lies in the $j$-th boundary component of $A$. The notion of state and pairing is defined as before.
\end{defn}

\subsection{Descending arrow diagrams}
In this subsection we define a special type of multi-based arrow diagrams.

\begin{defn}\rm
Let $A$ be a multi-based arrow diagram with $k$ boundary components. As we go
along the first boundary component of $\S(A)$ starting from the base point $*_1$, we pass
on the boundary of each ribbon once or twice. Then we continue to go along the second boundary component of $\S(A)$ starting from the base point $*_2$ and so on until we pass all boundary components of $\S(A)$. Arrow diagram $A$ is \textit{descending}
if we pass each ribbon of $\S(A)$ first time in the direction of its core arrow.
\end{defn}

\begin{rem} \rm
In order to define the notion of descending arrow diagrams we used the fact that all arrow diagrams are multi-based. The position of base points in an arrow diagram is essential to define an order of the passage.
\end{rem}

From now on we will work only with multi-based arrow diagrams.

\begin{ex}\rm
Arrow diagram with two boundary components in Figure \ref{fig:arrow_diagrams}a is descending and arrow diagram with three boundary components in Figure \ref{fig:arrow_diagrams}b is not descending.
\end{ex}

\begin{figure}[htb]
\centerline{\includegraphics[height=0.9in]{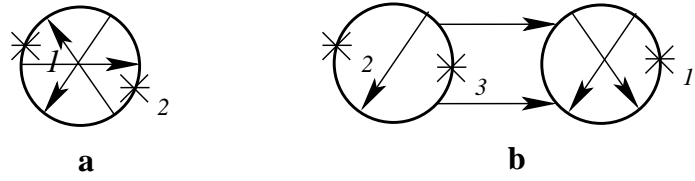}}
%\vspace{1in}
\caption{\label{fig:arrow_diagrams} Descending and non-descending arrow diagrams.}
\end{figure}

Denote by $\D_{n,k}$ the set of all
descending arrow diagrams with $n$ arrows and $k$ boundary components.
\begin{ex}\rm
The set $\D_{2,1}$ is presented below.
$$\D_{2,1}:=\quad\ris{-4}{-3}{80}{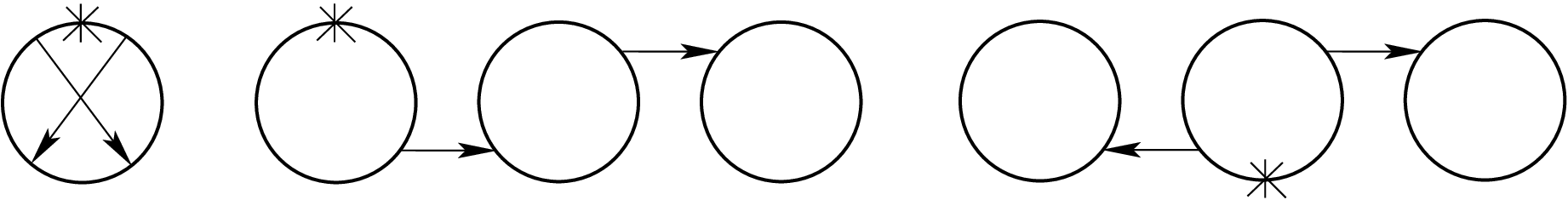}{-1.1}$$
\end{ex}

Let $G_m$ be any braid Gauss diagram. We denote by $w(G_m)$ the writhe of $G_m$, i.e. the sum of signs of all arrows in $G_m$. For each pair of natural numbers $j,k$ denote by
$$f_j^{(k)}(G_m):=\left(a^{-m-w(G_m)+1}(a^2-1)^{j-1}\right)^{(k-1)}|_{a=1}.$$
Let $G\in\mathfrak{G}_{k,m}$. A state $S(\phi)$ corresponding to $\phi:A\to G$ for a descending diagram $A$ with $k$ boundary
components will be also called \textit{descending}.

\begin{defn}\label{defn:P-k} \rm
For a pair $k,j$ such that $1\leq j\leq k$ set
$$D_{n,k,j}(G_m):=\sum_{G\in \mathfrak{G}_{j,m}}\sum_{A\in \D_{n+j-k,j}}\AG$$
and denote by
$$D_{n,k}(G_m):=\sum_{j=1}^k f_j^{(k)}(G_m)D_{n,k,j}(G_m).$$
Define the following polynomial:
$$P_k(G_m):=\sum_{n=0}^\infty D_{n,k}(G_m) z^n.$$
\end{defn}

\begin{rem}\label{rem:1-bd}\rm
Let $G_m$ be any braid Gauss diagram of a link $L$. Note that if $k=1$, then $f_1^{(1)}(G_m)=1$ and $D_{n,1}(G_m)=D_{n,1,1}(G_m)$ is exactly the sum with signs of all descending arrow diagrams with one base point and with one boundary component inside $G_m$. It follows by Theorem of Chmutov-Khoury-Rossi \cite{CKR} that $D_{n,1}(G_m)$ is the $n$-th coefficient $c_n(L)$ of the Conway polynomial $\nabla(L)$. Hence $P_1(G_m)$ is nothing but $\nabla(L)$.
\end{rem}

\section{Skein relation}

In this section we show that each $D_{n,k,j}(G_m)$ satisfies Conway skein relation for each $n$ and $1\leq j\leq k$. The fact that $D_{n,1,1}(G_m)$ satisfies Conway skein relation was proved in \cite{B,CKR}, i.e.

\begin{figure}[htb]
\centerline{\includegraphics[height=0.7in]{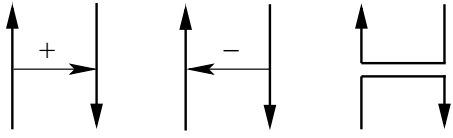}}
%\vspace{1in}
\caption{\label{fig:Conway} A Conway triple of Gauss diagrams.}
\end{figure}

\begin{thm}\cite{B,CKR}\label{thm:1-comp-skein-relation}
Let $G_{m,+}$, $G_{m,-}$ and $G_{m,0}$ be braid Gauss diagrams which differ
only in the fragment shown in Figure \ref{fig:Conway}. Then
\begin{equation}
P_1(G_{m,+})-P_1(G_{m,-})=zP_1(G_{m,0}).
\end{equation}
\end{thm}

Here we define a notion of a separating state. This notion will be used in the proof of Theorem \ref{thm:skein1}.
\begin{figure}[htb]
\centerline{\includegraphics[height=1.2in]{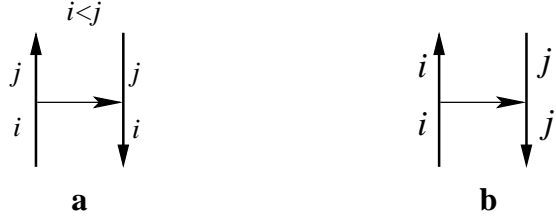}}
%\vspace{1in}
\caption{\label{fig:asc-des-arcs} Descending labeling.}
\end{figure}

\begin{defn}\rm
Let $G_m$ be a braid Gauss diagram and $G\in\mathfrak{G}_{k,m}$.
A \textit{descending separating state $S$ of $G$} is a state $S$ of $G$, together with a labeling of all arcs of $G$ by numbers from $1$ to $k$ such that:
\begin{itemize}
\item
An arc with a basepoint $*_i$ is labeled by $i$.
\item
Each arc near $\a\in S$ is labeled as in Figure \ref{fig:asc-des-arcs}a.
\item
Each arc near $\a\notin S$ is labeled as in Figure \ref{fig:asc-des-arcs}b.
\end{itemize}
\end{defn}

Let $G\in\mathfrak{G}_{k,m}$. Then every descending separating state $S$ in $G$ defines a new Gauss diagram $G_S$ with labeled circles as follows:\\
We smooth each arrow in $G$ which belongs to $S$, as shown in Figure \ref{fig:smoothing1}, and denote resulting smoothed Gauss diagram by $G_S$. Each circle in $G_S$ is labeled by $i$, if it contains an arc labeled by $i$.

\begin{figure}[htb]
\centerline{\includegraphics[height=0.75in]{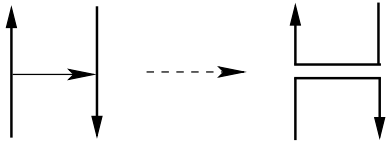}}
%\vspace{1in}
\caption{\label{fig:smoothing1} Smoothing of an arrow.}
\end{figure}

Now we return to arrow diagrams. Let $A\in\mathcal{D}_{n,k}$. We denote by $\sigma(A)$ the set of separating arrows in $A$ and label the arcs of circles in $A$ by $i$ if the corresponding arc belongs to the $i$-th boundary component of $\S(A)$. Note that for each $G\in\mathfrak{G}_{k,m}$ the homomorphism $\phi:A\rightarrow G$ induces a descending separating state $S$ of $G$, by taking $S=\phi(\sigma(A))$ and labeling each arc of $G$ by the same label as the corresponding arc of $A$.

\begin{defn} \rm
Let $S$ be a descending separating state of $G\in\mathfrak{G}_{k,m}$, $A\in\mathcal{D}_{n,k}$, and $\phi:A\rightarrow G$. We say that $\phi$ is \textit{$S$-admissible}, if a descending separating state induced by $\phi$ coincides with $S$.
\end{defn}

\begin{defn}\label{defn:sep-state-pairing} \rm
Let $S$ be a descending separating state of $G\in\mathfrak{G}_{k,m}$, and $A\in\mathcal{D}_{n,k}$. We define an \textit{$S$-pairing} $\AG_S$ by:
$$\AG_S:=\sum\limits_{\phi:A\rightarrow G}\s(\phi),$$
where the summation is over all $S$-admissible $\phi:A\rightarrow G$. We set
$$D_{n,k}(G)_S:=\sum_{A\in\D_{n,k}}\AG_S.$$
\end{defn}

Every descending separating state $S$ of $G$
defines Gauss diagrams $\{G^i_S\}_{i=1}^k$ as follows: $G^i_S$ consists of all circles of $G_S$ labeled by $i$, and its
arrows are arrows of $G$ with both ends on these circles. All arrows with ends
on circles of $G_S$ with different labels are removed. The base point of $G^i_S$ is
the base point $*_i$ of $G$.

Each $G^i_S$ corresponds to link $L^i_S$ which is defined as follows.
We smooth all crossings which correspond to arrows in $S$, as shown below:
\begin{equation*}
\ris{-4}{-3}{55}{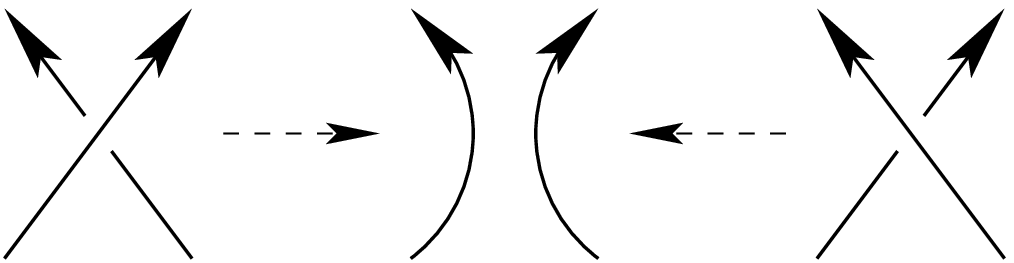}{-1.1}
\end{equation*}
We obtain a diagram of a smoothed link $L_S$ with labeling of components
induced from the labeling of circles of $G_S$. Denote by $L^i_S$ a
sublink which consists of components labeled by $i$.

It follows from Remark \ref{rem:1-bd} that for every $n\geq 0$
we have
$$D_{n,1}(G^i_S)=c_n(L^i_S).$$
Using this together with the definition of $D_{n,k}(G)_S$ we get

\begin{lem}\label{lem:sep-states}
Let $G_m$ be a braid Gauss diagram of a link $L$ and $G\in\mathfrak{G}_{k,m}$
Then for every $n\geq 0$ and a descending separating state $S$ of $G$ we have
\begin{equation*}
D_{n,k}(G)_S=\s(S)\sum\prod_{j=1}^k c_{i_j}(L^j_S),
\end{equation*}
where the summation is over indices $i_1,...,i_k$ such that $\sum_{j=1}^k i_j=n-|S|$.
Here $|S|$ is the number of arrows in $S$ and $\s(S)=\prod_{\a\in S}\s(\a)$.
\end{lem}

Summing over all $G\in\mathfrak{G}_{j,m}$ and over all descending separating states $S$ of $G$, we obtain
\begin{cor}\label{cor:sep-states}
Let $G_m$ be a braid Gauss diagram of a link $L$.
Then for every $n\geq 0$ and $1\leq j\leq k$
\begin{align*}
&D_{n,k,j}(G_m)=\sum_{G\in\mathfrak{G}_{j,m}}\sum_{l=0}^{n+j-k}\sum_{S,|S|=l}D_{n+j-k,j}(G)_S\quad \rm{or}\\
&D_{n,k,j}(G_m)=\sum_{G\in\mathfrak{G}_{j,m}}\sum_{l=0}^{n+j-k}\sum_{S,|S|=l}\s(S)\sum_{i_1,...,i_j}\prod_{t=1}^j c_{i_t}(L^t_S),
\end{align*}
where the third summation is over all descending
separating states $S$ of $G\in\mathfrak{G}_{j,m}$, and the fourth summation is over all indices $i_1,...,i_j$ such that $\sum_{t=1}^j i_t=n+j-k-l$.
\end{cor}

At this point we establish the skein relation for $D_{n,k,j}(G_m)$.
\begin{thm}\label{thm:skein1}
Let $G_{m,+}$, $G_{m,-}$, $G_{m,0}$ be a Conway triple of braid Gauss diagrams, see Figure \ref{fig:Conway}.  Then
\begin{equation}\label{eq:A1}
D_{n,k,j}(G_{m,+})-D_{n,k,j}(G_{m,-})=D_{n-1,k,j}(G_{m,0})
\end{equation}
\end{thm}

\begin{proof}
Let $G_+\in(\mathfrak{G}_+)_{j,m}$, $G_-\in(\mathfrak{G}_-)_{j,m}$, $G_0\in(\mathfrak{G}_0)_{j,m}$ be a corresponding triple of colored
braid Gauss diagrams, i.e. if an arc in $G_+$ or in $G_-$ or in $G_0$ is labeled by some $a_i$, then it is colored by the same number or a basepoint in view of Definition \ref{defn:labeling}. Denote the arrows of $G_+$ and $G_-$ appearing in Figure \ref{fig:Conway}
by $\a_+$ and $\a_-$, respectively.

\begin{figure}[htb]
\centerline{\includegraphics[width=4in]{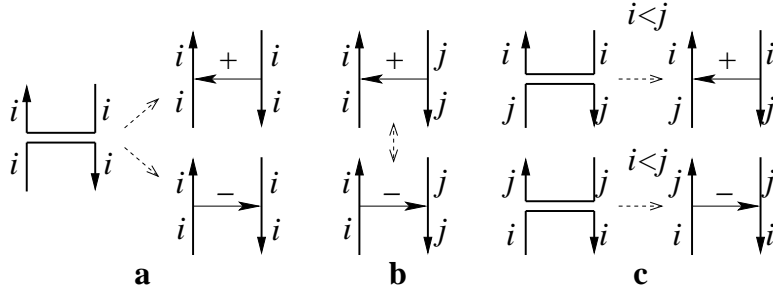}}
\caption{\label{fig:skein-sep-states} Correspondence of separating states of $G_0$ and $G_\pm$.}
\end{figure}

Let us look at labels of descending separating states of $G_\pm$ and $G_0$ on
four arcs of the shown fragment. If labels of all four arcs are the same, we may
identify states of $G_\pm$ and $G_0$ with the same arrows and labels of arcs,
see Figure \ref{fig:skein-sep-states}a. Lemma \ref{lem:sep-states} and
Theorem \ref{thm:1-comp-skein-relation} imply, that for every such state $S$
$$D_{n+j-k,j}(G_+)_S-D_{n+j-k,j}(G_-)_S=D_{n+j-k-1,j}(G_0)_S.$$
If labels on two arcs near the head of $\a_\pm$ coincide, but differ from
labels near the tail of $\a_\pm$, by Lemma \ref{lem:sep-states} we
have $D_{n+j-k,j}(G_+)_S-D_{n+j-k,j}(G_-)_S=0$ for any such state $S$
of $G_\pm$, and there is no corresponding state of $G_0$. See Figure
\ref{fig:skein-sep-states}b (for $i\neq j$).

There is one more case when labels of two arcs near the head of
$\a_\pm$ are different. Such a state $S$ of $G_0$ corresponds either
to a descending separating state $S\cup\a_+$ of $G_+$, or to a
descending separating state $S\cup\a_-$ of $G_-$, see Figure
\ref{fig:skein-sep-states}c. By Lemma \ref{lem:sep-states} we have
$D_{n+j-k,j}(G_+)_{S\cup\a_+}=D_{n+j-k-1,j}(G_0)_S$ in the first case
and $D_{n+j-k,j}(G_-)_{S\cup\a_-}=-D_{n+j-k-1,j}(G_0)_S$ in the second case.
Summing over all $G_+\in(\mathfrak{G}_+)_{j,m}$, $G_-\in(\mathfrak{G}_-)_{j,m}$, $G_0\in(\mathfrak{G}_0)_{j,m}$, over all descending separating states of $G_+$, $G_-$, $G_0$ and using Corollary \ref{cor:sep-states},
we obtain the statement of the theorem.
\end{proof}

For a pair $k,j$ such that $1\leq j\leq k$ set
$$A_{k,j}(G_m):=\sum_{n=0}^\infty D_{n,k,j}(G_m) z^{n+j-k}.$$ It follows from Definition \ref{defn:P-k} that
\begin{equation}\label{eq:P-k}
P_k(G_m)=\sum_{j=1}^k f_j^{(k)}(G_m)z^{k-j}A_{k,j}(G_m).
\end{equation}
Moreover, the following Corollary follows immediately from Theorem \ref{thm:skein1}.

\begin{cor}\label{cor:skein-A-k-j}
Let $G_{m,+}$, $G_{m,-}$, $G_{m,0}$ be a Conway triple of braid Gauss diagrams, then
$$A_{k,j}(G_{m,+})-A_{k,j}(G_{m,-})=zA_{k,j}(G_{m,0}).$$
\end{cor}

\section{Main theorem}
For a link $L$ and $k\geq 0$ denote by $\P_k(L):=z^kP_a^{(k)}(L)|_{a=1}$.

\begin{thm}\label{thm:main}
Let $G_m$ be a braid Gauss diagram of a link $L$, then for $k\geq 0$
$$P_{k+1}(G_m)=\P_{k}(L).$$
\end{thm}
We will prove this theorem at the end of the section. At this point we show that the polynomials $P_{k+1}$ and $\P_k$ satisfy the same skein relation. The skein relation for the polynomial $\P_k$ follows directly from the skein relation of HOMFLY-PT polynomial, i.e.
\begin{equation}\label{eq:skein-P}
\begin{split}
\P_k(L_+)-\P_k(L_-)+kz\P_{k-1}(L_+)+\sum_{i=0}^{k-1}(-1)^{k-1-i}\frac{k!}{i!}z^{k-i}\P_i(L_-)=zI_k(L_0).
\end{split}
\end{equation}

\begin{lem}
Let $G_m$ be a braid Gauss diagram of a link $L$. Then for every $k\geq 0$ we have
\begin{equation}\label{eq:skein-G}
\begin{split}
&P_{k+1}(G_{m,+})-P_{k+1}(G_{m,-})=\\
&zP_{k+1}(G_{m,0})-kzP_k(G_{m,+})-\sum_{i=0}^{k-1}(-1)^{k-1-i}\frac{k!}{i!}z^{k-i}P_{i+1}(G_{m,-}).
\end{split}
\end{equation}
\end{lem}

\begin{proof}
It follows from \eqref{eq:P-k} and Corollary \ref{cor:skein-A-k-j} that
\begin{equation}\label{eq:skein-part1}
\begin{split}
&P_{k+1}(G_{m,+})-P_{k+1}(G_{m,-})-zP_{k+1}(G_{m,0})=\\
&-\sum_{j=1}^{k+1}\left(f_j^{(k+1)}(G_{m,0})-f_j^{(k+1)}(G_{m,+})\right)z^{k+1-j}A_{k+1,j}(G_{m,+})-\\
&\sum_{j=1}^{k+1}\left(f_j^{(k+1)}(G_{m,-})-f_j^{(k+1)}(G_{m,0})\right)z^{k+1-j}A_{k+1,j}(G_{m,-}).
\end{split}
\end{equation}
Now
\begin{align*}
&f_j^{(k+1)}(G_{m,0})-f_j^{(k+1)}(G_{m,+})=\left(a^{-w(G_{m,+})-m+2}(a^2-1)^{j-1}\right)^{(k)}|_{a=1}-\\
&\left(a^{-w(G_{m,+})-m+1}(a^2-1)^{j-1}\right)^{(k)}|_{a=1}=kf_j^{(k)}(G_{m,+}).
\end{align*}
Note that $f_{k+1}^{(k)}(G_{m,+})=0$ and $A_{k+1,j}(G_{m,+})=A_{k,j}(G_{m,+})$. Hence we obtain
\begin{equation}\label{eq:G+}
\sum_{j=1}^{k+1}\left(f_j^{(k+1)}(G_{m,0})-f_j^{(k+1)}(G_{m,+})\right)z^{k+1-j}A_{k+1,j}(G_{m,+})=kzP_k(G_{m,+}).
\end{equation}
Similarly we have
\begin{align*}
&f_j^{(k+1)}(G_{m,-})-f_j^{(k+1)}(G_{m,0})=\left(a^{-w(G_{m,-})-m+1}(a^2-1)^{j-1}\right)^{(k)}|_{a=1}-\\
&\left(a^{-w(G_{m,-})-m}(a^2-1)^{j-1}\right)^{(k)}|_{a=1}=\sum_{i=0}^{k-1}(-1)^{k-1-i}\frac{k!}{i!}f_j^{(i+1)}(G_{m,-}).
\end{align*}
It follows that
\begin{equation}\label{eq:skein-part2}
\begin{split}
&\sum_{j=1}^{k+1}\left(f_j^{(k+1)}(G_{m,-})-f_j^{(k+1)}(G_{m,0})\right)z^{k+1-j}A_{k+1,j}(G_{m,-})= \\
&\sum_{j=1}^{k+1}\sum_{i=0}^{k-1}(-1)^{k-1-i}\frac{k!}{i!}f_j^{(i+1)}(G_{m,-})z^{k+1-j}A_{k+1,j}(G_{m,-})=\\
&\sum_{i=0}^{k-1}(-1)^{k-1-i}\frac{k!}{i!}z^{k-i}\sum_{j=1}^{i+1}f_j^{(i+1)}(G_{m,-})z^{i+1-j}A_{i+1,j}(G_{m,-})=\\
&\sum_{i=0}^{k-1}(-1)^{k-1-i}\frac{k!}{i!}z^{k-i}P_{i+1}(G_{m,-}),
\end{split}
\end{equation}
where the second equality follows from the fact that for $j>i+1$ we have $f_j^{(i+1)}(G_{m,-})=0$, and for $i<k$ we have $A_{k+1,j}(G_{m,-})=A_{i+1,j}(G_{m,-})$. Combining equalities \eqref{eq:skein-part1}, \eqref{eq:G+} and \eqref{eq:skein-part2} we conclude the proof of the lemma.
\end{proof}

Let $G_m$ be a braid Gauss diagram with $r$ circles. Of course $r\leq m$. Recall that $G_m$ contains $m$ arcs labeled by letters $a_1,...,a_m$. Let
$\{m_i\}_{i=1}^r$ be a unique subset of the set $\{i\}_{i=1}^m$ which is defined as follows:
\begin{itemize}
\item Each circle $C$ in $G_m$ contains $k_c$ arcs labeled by $a_{c_1},...,a_{c_{k_c}}$. Let $m_c=\min\{a_{c_1},...,a_{c_{k_c}}\}$.
There are exactly $r$ such numbers, i.e. one for each circle. We place them in an ascending order and the $i$-th number in this ascending sequence is denoted by $a_{m_i}$. In particular, it follows that $a_1=a_{m_1}<a_{m_2}<...\leq a_m$.
\end{itemize}
The sequence $\{m_i\}_{i=1}^r$ defines an obvious ordering of circles of $G_m$, i.e. the circle of $G_m$ which contains an arc labeled by $a_{m_i}$ is called the $i$-th circle.

\begin{defn}\rm
A braid Gauss diagram $G_m$ is called \emph{totally ascending}, if for each pair $i,j$ such that $1\leq i\leq j\leq r$ the following holds.
\begin{itemize}
\item When we walk on the $i$-th
circle of $G_m$ starting from an arc labeled by $a_{m_i}$ until we
return to this arc, we pass all arrows that connect $i$-th and $j$-th circles
first at the arrowhead.
\end{itemize}
\end{defn}

\begin{rem*}\rm
Note that if a totally ascending braid Gauss diagram $G_m$ on $r$ circles represents a link $L$. Then $L$ is an $r$-component unlink $O_r$.
\end{rem*}

\begin{lem}\label{lem:tot-asc-equality}
Let $G_m$ be a totally ascending braid Gauss diagram of an $r$-component unlink $O_r$. Then
$$r=m+w(G_m).$$
\end{lem}
\begin{proof}
The proof of this lemma may be obtained by induction on $m$. It is elementary and is left to the reader.
\end{proof}

\begin{cor}\label{cor:tot-ascending}
Let $G_m$ be a totally ascending braid Gauss diagram of an $r$-component unlink $O_r$. Then for every $k\geq 0$ we have
$$P_{k+1}(G_m)=\P_k(O_r).$$
\end{cor}

\begin{proof}
By definition
\begin{equation*}
\P_k(O_r)=\left\{\begin{array}{c}
\begin{aligned}
&\left(a^{-r+1}(a^2-1)^{r-1}\right)^{(k)}|_{a=1}\cdot z^{k-r+1}\hspace{3mm}\rm{if}\hspace{2mm} &1\leq r\leq k+1&\\
&\hspace{2.5cm}0  &\rm{otherwise}\thinspace.\\
\end{aligned}
\end{array}\right.
\end{equation*}
The diagram $G_m$ is totally ascending, hence
\begin{equation*}
D_{n,k+1,j}(G_m)=\left\{\begin{array}{c}
\begin{aligned}
&1\quad  \textrm{if}\quad j=r\hspace{3mm} \textrm{and}\hspace{3mm} &\hspace{2mm}n+r-k-1=0&\\
&0  &\rm{otherwise}\thinspace.\\
\end{aligned}
\end{array}\right.
\end{equation*}
It follows that
\begin{equation*}
P_{k+1}(G_m)=\left\{\begin{array}{c}
\begin{aligned}
&f_r^{(k+1)}(G_m)\cdot z^{k-r+1}\quad  \rm{if}\quad &1\leq r\leq k+1&\\
&\hspace{2.5cm}0  &\rm{otherwise}\thinspace.\\
\end{aligned}
\end{array}\right.
\end{equation*}
By definition $f_r^{(k+1)}(G_m):=\left(a^{-m-w(G_m)+1}(a^2-1)^{r-1}\right)^{(k)}|_{a=1}$. Lemma \ref{lem:tot-asc-equality} states that $-m-w(G_m)=-r$ and the proof follows.
\end{proof}

Now we are ready to prove our main theorem.

\begin{proof}[Proof of Theorem \ref{thm:main}]
We prove this theorem by induction on $k$.
If $k=0$, then by Remark \ref{rem:1-bd} $P_1(G_m)=\P_0(L):=\nabla(L)$.

Now let us assume that $P_k(G'_m)=\P_{k-1}(L')$ for any $G'_m$ which represents a link $L'$.
We have to show that for any $G_m$ which represents some link $L$ we have $P_{k+1}(G_m)=\P_k(L)$.
We prove this statement by induction on the number of arrows of the braid Gauss diagram $G_m$ of an $r$-component link~$L$.

If $G_m$ has no arrows, then it represents an $r$-component unlink, it is totally ascending and by Corollary \ref{cor:tot-ascending} we have $P_{k+1}(G_m)=\P_k(O_r)$.

Let us assume that $P_{k+1}(\widetilde{G}_m)=\P_k(\widetilde{L})$ for each link $\widetilde{L}$ and every braid Gauss diagrams $\widetilde{G}_m$ with less than $d$ arrows, which represent $\widetilde{L}$.

Let $G_m$ be a diagram with $d$ arrows. We can pick an arrow in $G_m$ and use the skein relation \eqref{eq:skein-G} to
simplify $G_m$. By induction hypothesis, each $P_i(G_{m,-})$, $P_k(G_{m,+})$ and $P_{k+1}(G_{m,0})$ in the right hand side of \eqref{eq:skein-G}
coincides with $\P_{i-1}(G_{m,-})$, $\P_{k-1}(L_+)$ and $\P_k(L_0)$ respectively.

We can make our braid Gauss diagram $G_m$ totally ascending by changing the appropriate arrows using the relation \eqref{eq:skein-G}. Hence we can represent $P_{k+1}(G_m)$ as $P_{k+1}(G'_m)$, for some totally ascending braid Gauss diagram $G'_m$, plus some terms of the
form $P_{k+1}(G_{m,0})$, where $G_{m,0}$ has less than $d$ arrows, and plus some terms of the
form $P_{i}(G''_m)$ for $1\leq i\leq k$, where $G''_m$ has $d$ arrows. The diagram $G'_m$ represents an $r$-component unlink $O_r$ and by Corollary \ref{cor:tot-ascending} we have $P_{k+1}(G'_m)=\P_k(O_r)$. By \eqref{eq:skein-P}, \eqref{eq:skein-G} and the induction hypothesis
$$P_{k+1}(G_m)=\P_k(L).$$
\end{proof}

\begin{ex}\rm
Let $G_2$ be a braid Gauss diagram of the trefoil knot $T$ shown in Figure \ref{fig:tref-Gauss} and let $G^2$ and $G^1$ be the unique colored braid Gauss diagrams in the sets $\mathfrak{G}_{2,2}$ and $\mathfrak{G}_{1,2}$ shown in Figures \ref{fig:braid-tref-states}a and \ref{fig:braid-tref-states}b respectively. We are going to calculate $P_3(G_2)$.\\
\begin{figure}[htb]
\centerline{\includegraphics[height=1.6in]{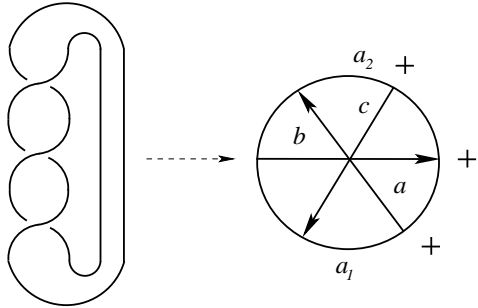}}
%\vspace{0.1in}
\caption{\label{fig:tref-Gauss} Trefoil $T$, braid Gauss diagram $G_2$ with labeled arrows.}
\end{figure}

\begin{figure}[htb]
\centerline{\includegraphics[height=1.4in]{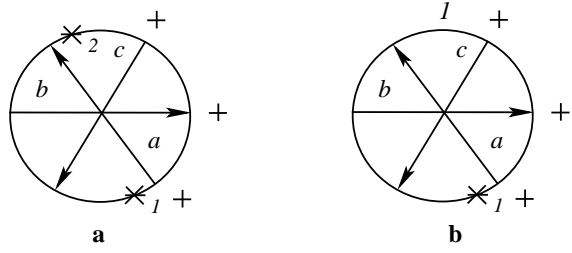}}
%\vspace{0.1in}
\caption{\label{fig:braid-tref-states} Colored braid Gauss diagrams $G^2$ and $G^1$.}
\end{figure}

Recall that $D_{n,3,j}(G_2):=\sum_{G\in \mathfrak{G}_{j,2}}\sum_{A\in \D_{n+j-3,j}}\AG$ for each $1\leq j\leq 3$, and
$D_{n,3}(G_2)=\sum_{j=1}^3 f_j^{(3)}(G_2)D_{n,3,j}(G_2)$. Note that $D_{n,3,j}(G_2)=0$ for $j>2$ because in this case the set $\mathfrak{G}_{j,2}$ is empty. In edition, we have
$$D_{n,3,1}(G_2)=\sum_{A\in \D_{n-2,1}}\langle A,G^1\rangle\quad \textrm{and} \quad D_{n,3,2}(G_2)=\sum_{A\in \D_{n-1,2}}\langle A,G^2\rangle.$$
There is a unique descending state of $G^1$ with $0$ arrows. The only other descending state of $G^1$ is $\{a,b\}$. It follows that
$D_{2,3,1}(G_2)=D_{4,3,1}(G_2)=1$ and $D_{n,3,1}(G_2)=0$ if $n\neq 2,4$. The only descending states of $G^2$ with $1$ arrow are $\{a\}$ and $\{c\}$, and the only descending state of $G^2$ with $3$ arrows is $\{a,b,c\}$. Hence $D_{2,3,2}(G_2)=2$, $D_{4,3,2}(G_2)=1$ and $D_{n,3,2}(G_2)=0$ if
$n\neq 2,4$. It follows that $D_{n,3}(G_2)=0$ if $n\neq 2,4$. In order to calculate $D_{2,3}(G_2)$ and $D_{4,3}(G_2)$ we need to compute $f_1^{(3)}(G_2)$ and $f_2^{(3)}(G_2)$. We have $w(G_2)=3$, hence
$$f_1^{(3)}(G_2):=(a^{-4})^{(2)}|_{a=1}=20\quad\textrm{and}\quad f_2^{(3)}(G_2):=(a^{-4}(a^2-1))^{(2)}|_{a=1}=-14.$$
It follows that $D_{2,3}(G_2)=\sum_{j=1}^3 f_j^{(3)}(G_2)D_{2,3,j}(G_2)=20\cdot 1-14\cdot 2=-8$ and
$D_{4,3}(G_2)=\sum_{j=1}^3 f_j^{(3)}(G_2)D_{4,3,j}(G_2)=20\cdot 1-14\cdot 1=6$.
Hence $P_3(G_2)=6z^4-8z^2$. Indeed, one may check that
$P(T)=a^{-2}z^2+2a^{-2}-a^{-4}$, so $\P_2(T):=z^2P^{(2)}_a(T)|_{a=1}=6z^4-8z^2$.

\end{ex}

\section{Final Remarks}

\textbf{1.} It is possible to show that for each $k\geq 0$ the polynomial $P_{k+1}(G_m)$ is invariant of an underlying link without using the HOMFLY-PT polynomial. Recall that by Markov's theorem two braids $\a$ and $\b$ represent the same link if and only if $\a$ can be obtained from $\b$ by a finite sequence of Markov moves.
\begin{itemize}
\item
Conjugation in the braid group, i.e. replace $\a\in B_m$ by $\g\a\g^{-1}$ where $\g\in B_m$.
\item
Stabilization move, i.e. replace $\a\in B_m$ by $\a\sig_n$ or by $\a\sig_n^{-1}$. Destabilization move, i.e. perform the converse operation.
\end{itemize}

In the language of braid Gauss diagrams it means that we have to prove the invariance of $P_{k+1}(G_m)$ under moves shown in Figure~\ref{fig:braid-Gauss-Rmoves}. The proof is very similar to the proof of main theorem in \cite{B}.

\textbf{2.} We have defined $P_{k+1}(G_m)$ in terms of counting descending arrow diagrams in a braid Gauss diagram $G_m$. Alternatively, we can define $P_{k+1}(G_m)$ in terms of counting ascending arrow diagrams in a braid Gauss diagram~$G_m$.

\begin{defn}\label{defn:asc-labeling}\rm
Let $k\geq 1$ be any integer and $G_m$ a braid Gauss diagram associated with a braid $\a\in B_m$. A \emph{$\star$-colored braid Gauss diagram} $G^\star_{k,m}$ is a diagram $G$ together with the following assignment of $k$ base points $*_1,...,*_k$ and $m-k$ natural numbers between $1$ and $k$ to the arcs labeled by letters $a_1,...,a_m$:
\begin{itemize}
\item
For each $1\leq i\leq k$ there exists \textbf{exactly one} arc $a_j$ such that the base point $*_i$ is placed on this arc, and the arc $a_m$ always contains basepoint $*_1$.
\item
Let $1\leq i_1<i_2\leq k$. If $*_{i_1}$ and $*_{i_2}$ are placed on arcs $a_{j_1}$ and $a_{j_2}$, then $j_1>j_2$, i.e. the assignment of base points is in descending order.
\item
After we placed base points $*_1,...,*_k$ on arcs $a_{j_1},...,a_{j_k}$, let $a_j$ be a non-based arc. Denote by $j_l$ the minimal number from the set $\{j_1,...,j_k\}$ such that $j<j_l$. Now, to arc $a_j$ we assign \textbf{exactly one} number from the set $\{1,...,l\}$.
\end{itemize}
\end{defn}

We denote by $\mathfrak{G}^\star_{k,m}$ a set of all $\star$-colored braid Gauss diagrams associated with $k$ and $G_m$.

\begin{defn}\rm
Let $A$ be a multi-based arrow diagram with $k$ boundary components. As we go
along the first boundary component of $\S(A)$ starting from the base point $*_1$, we pass
on the boundary of each ribbon once or twice. Then we continue to go along the second boundary component of $\S(A)$ starting from the base point $*_2$ and so on until we pass all boundary components of $\S(A)$. Arrow diagram $A$ is \textit{ascending}
if we pass each ribbon of $\S(A)$ first time in the opposite direction of its core arrow. The notion of state and pairing is defined as in Subsection \ref{subsec-G-diag}.
\end{defn}

\begin{ex}\rm
Arrow diagram with three boundary components shown below is ascending.
\begin{equation*}
\ris{-4}{-3}{60}{3-based-a-diagram}{-1.1}
\end{equation*}
\end{ex}

Denote by $\A_{n,k}$ the set of all ascending arrow diagrams with $n$ arrows and $k$ boundary components.
Let $G_m$ be any braid Gauss diagram. For a pair $k,j$ such that $1\leq j\leq k$ set
$$A_{n,k,j}(G_m):=\sum_{G\in \mathfrak{G}^\star_{j,m}}\sum_{A\in \A_{n+j-k,j}}\AG.$$
Denote
$$A_{n,k}(G_m):=\sum_{j=1}^k f_j^{(k)}(G_m)A_{n,k,j}(G_m)\hspace{1.4cm} P^\star_k(G_m):=\sum_{n=0}^\infty A_{n,k}(G_m) z^n.$$

\begin{thm}
Let $G_m$ be a braid Gauss diagram of a link $L$, then for $k\geq 0$
$$P^\star_{k+1}(G_m)=\P_{k}(L).$$
\end{thm}
The proof of this theorem is identical to the proof of our main theorem and is left to the reader.

\bibliographystyle{alpha}

\bigskip

\emph{Department of Mathematics, Vanderbilt University, Nashville, TN 37240}

\emph{E-mail address:} \verb"michael.brandenbursky@Vanderbilt.Edu"

\end{document}